\documentclass{article}
\usepackage{amsmath}
\usepackage{amsfonts}
\usepackage{amssymb}
\usepackage{amsthm}
\usepackage{stackrel}
\usepackage{graphicx}
\usepackage{tikz}
\usepackage{xcolor}
\usepackage[export]{adjustbox}
\usetikzlibrary{arrows,positioning, shapes} 
\usepackage{float}
\usepackage[utf8]{inputenc}
 \usepackage[T1]{fontenc}
\usepackage{url}

\usepackage{figsize}

\usepackage{graphics}

\newtheorem{ex}{Example}[section]

\newtheorem{lm}{Lemma}[section]
\newtheorem{df}{Definition}[section]
\newtheorem{tm}{Theorem}[section]
\newtheorem{co}{Corollary}[section]
\newtheorem{re}{Remark}[section]

\newcommand{\R}{{\rm I}\kern-0.18em{\rm R}}
\newcommand{\1}{{\rm 1}\kern-0.25em{\rm I}}
\newcommand{\E}{{\rm I}\kern-0.18em{\rm E}}
\newcommand{\p}{{\rm I}\kern-0.18em{\rm P}}
\makeatletter

\def\@fnsymbol#1{\ensuremath{\ifcase#1\or a\or b\or c\or d\or \e\or f\or *\dagger 	\or \ddagger\ddagger \else\@ctrerr\fi}}

\author{Vesna Gotovac \DJ{}ogaš\footnote{Faculty of Science, University of Split}}

\date{}

\begin{document}
\title{On some topological properties of normed Boolean algebras}

\author{Vesna Gotovac \DJ{}oga\v{s}} 

\markboth{V. Gotovac \DJ{}oga\v{s}}

\maketitle

\begin{abstract} 

This paper concerns the compactness and separability properties of the normed Boolean algebras (N.B.A.) with respect to topology generated by a distance equal to the square root of a measure of symmetric difference between two elements. The motivation arises from studying random elements talking values in N.B.A.
 Those topological properties are important assumptions that enable us to avoid possible difficulties when generalising concepts of random variable convergence, the definition of conditional law and others.

For each N.B.A., there exists a finite measure space $(E, {\mathcal E}, \mu)$ such that the N.B.A. is isomorphic to $(\widetilde{\mathcal E}, \widetilde{\mu})$ resulting from the factorisation of initial $\sigma$-algebra by the ideal of negligible sets. We focus on topological properties of $(\widetilde{\mathcal E}, \widetilde{\mu})$ in general setting when $\mu$ can be an infinite measure. In case when $\mu$ is infinite, we also consider properties of  $\widetilde{\mathcal E}_{fin} \subseteq \widetilde{\mathcal E}$ consisting of classes of measurable sets having finite measure.    

The compactness and separability of the N.B.A. are characterised using the newly defined terms of approximability and uniform approximability of the corresponding measure space.

Finally,  conditions  on $(E,\mathcal E,\mu)$ are derived for separability and compactness of $\widetilde{\mathcal E}$ and $\widetilde{\mathcal E}_{fin}.$ 

\end{abstract}
{\bf keywords:  compact, locally compact, Polish space, separable}

\section{Introduction}\label{sec1} 
\setcounter{equation}{0}
The motivation for studying the topological properties of normed Boolean algebras arises from probability theory, more precisely from its subfield of stochastic geometry.
Nowadays, the mathematical theory of random sets is very popular. 
The books \cite{Ma} and \cite{Mo} provide basic definitions, notions and theoretical results on closed random sets (or compact random sets).
An approach to defining a random set that takes values in a more general family of sets than closed or compact sets is presented in \cite{Rs}. There, the random set is represented as a random element taking values in a normed Boolean algebra (N.B.A.), i.e. a complete Boolean algebra endowed with a strictly positive finite measure, see \cite{Vl}.)  These random elements are defined using Borel subsets of N.B.A. generated by a distance on N.B.A. equal to the square root of a measure of symmetric difference between two elements.


If we want to study different types of convergence of these random sets taking values in N.B.A. or generalise some other concepts related to random variables, it is beneficial to ensure that the space of its values
is a Polish space or a locally compact, Hausdorﬀ and second countable topological space (LCSH space). This motivated us to study the topological properties of the N.B.A.s with respect to topology generated by the distance equal to the square root of a measure of symmetric difference between two elements.

Let us mention some conveniences we get when working with random elements with a separable metric space of values. In this setting, for every two random elements $X$ and $X',$ a set $\{ X=X'\}$ is an event. The distance between two random elements is a random variable, which allows us to introduce convergence in probability (see \cite{War}). In this case, the space of simple random elements is a dense subspace. If the space of values of the random elements is complete and separable (Polish), then the conditional law can be defined and the Doob-Dynkin representation holds (see \cite{Polish}).

Locally compactness is also a desirable property when considering weak convergence of distributions of random elements (see \cite{Weak_convergence}).


For each complete N.B.A. $({\mathcal X},m)$ where $m$ is finite, there exists a finite measure space $(E, {\mathcal E}, \mu)$ such that the N.B.A. $({\mathcal X},m)$ and N.B.A. $(\widetilde{\mathcal E}, \widetilde{\mu})$ resulting from factorisation of initial $\sigma$-algebra by the ideal of negligible sets are isomorphic (see \cite{Vl}). Following this result we derive that the N.B.A. is homeomorphic to the space of indicator functions in $L^p(E,\mathcal E, \mu).$ 
We generalise this setting allowing $\mu$ and corresponding $m$ to obtain infinite values. In this case, the above mentioned homeomorphism does not hold. 

As we mentioned before, if $\mu$ is finite, then topological properties of N.B.A. are equivalent to topological properties of a subset of indicators in $L^2$ space.
Following results concerning the separability of $L^p$ spaces are established. If $\mu$ is $\sigma$-finite and $\mathcal E$ is countably generated, then $L^{p}(E, \mathcal{E}, \mu)$ is separable for $1 \leq p<+\infty$ (see \cite[Proposition 3.4.5.]{MT}).
Since every metric subspace of separable metric space is separable \cite[Theorem VIII, p.~160 ]{Zo} if these conditions hold the space of indicators is separable as well. 

If measure $\mu$ is not finite, then $(\widetilde{\mathcal E},\widetilde{\mu})$ is not homeomorphic to the space of indicator functions in $L^p(E,\mathcal E, \mu).$  In this case we also consider $\widetilde{\mathcal E}_{fin}=\{ [A]:\mu(A)<\infty\} \subset \widetilde{\mathcal E},$ which is homeomorphic to the space of indicator functions in $L^p(E,\mathcal E, \mu).$ 
In case $(E,\mathcal E)=(\mathbb R^d,\mathfrak{B}(\mathbb R^d))$ we  prove the $\widetilde{\mathcal E}_{fin}$  and corresponding space of indicators is a separable  if measure $\mu$ is outer regular. Although $\mathfrak{B}(\mathbb R^d)$ is countably generated, there are measures on $\mathfrak{B}(\mathbb R^d)$ which are outer regular but not $\sigma$-finite.


The compactness of subsets of $L^p$-spaces has already been well studied, and some conditions for the compactness of general bounded subsets of $L^p$-spaces can be found in \cite{Ch} and \cite[Theorems 18,20,21 pp.297]{Du}. Although these conditions can be verified for our case when $\mu$ is finite, we introduce conditions that are easier to verify, more intuitive in our setting and can be applied for verifying compactness of $\widetilde{\mathcal E}$ in case when $\mu$ is infinite.

It is well known that a separable space is a space that is "well approximated by a countable subset" and a compact space is a space that is "well approximated by a finite subset".
 We construct conditions for the corresponding measure space that follow this intuition. We call those conditions approximability and uniform approximability.  We prove that if the measure can be well approximated by its values on a countable family or a finite family of measurable sets, then the corresponding N.B.A. is separable or a compact metric space, respectively. 
 
 
 Verifying the conditions of approximability and uniform approximability, we derive a conditions and in some cases characterisation for separability and compactness of $\widetilde{\mathcal E}$ and $\widetilde{\mathcal E}_{fin}$ based on properties of corresponding measure space $(E,\mathcal E,\mu).$ 
 
 
The outline of the paper is as follows. 

In the Preliminaries section we recall basic definitions and results concerning separability and compactness, we also mention some results from the measure theory we use for deriving results. The final subsection is dedicated to the terminology concerning Boolean algebras. The metric spaces $(\widetilde{\mathcal E},d_{\mu})$ and $(\widetilde{\mathcal E}_{fin},d_{\mu})$ are introduced and their completeness is discussed.

In the Main result section, we introduce properties of approximability and uniform approximability of measure with respect to a filtration. Separability and compactness are characterised using these terms. Further, we discuss  separability and compactness of  $(\widetilde{\mathcal E}),d_{\mu})$ and $(\widetilde{\mathcal E}_{fin},d_{\mu})$ based on the properties of the corresponding measure space $(E,\mathcal E,\mu).$

The paper is concluded by the Discussion section where the obtained results are summarised.
\section{Preliminaries}\label{sec1} 
\setcounter{equation}{0}
\subsection{Topological properties}
Let us first recall definitions and the basic relation of topological properties we study. The definitions and the results we present can be found in \cite{Topo} and \cite{Wi} .

For some $A \subset X,$ let $\mathcal{A}=\left\{G_{i}\right\}$ be a class of subsets of $X$ such that $A \subset \cup_{i} G_{i}.$  $\mathcal A$ is  called a \textbf{cover} of $A$, and an \textbf{open cover} if each $G_{i}$ is open. Furthermore, if a finite subclass of $A$ is also a cover of $A$, i.e. if
$ G_{i_{1}}, \ldots, G_{i_{m}} \in A$ such that $A \subset G_{i_{1}} \cup \cdots \cup G_{i_{m}}$
then $\mathcal A$  contains a \textbf{finite subcover}.

\begin{df}
 A subset $A$ of a topological space $X$ is \textbf{compact} if every open cover of $A$ is reducible to a finite cover.
\end{df}
In other words, if $A$ is compact and $A \subset \cup_{i} G_{i}$, where the $G_{i}$ are open sets, then one can select a finite number of the open sets $G_{i_{1}}, \ldots, G_{i_{m}}$, so that $A \subset G_{i_{1}} \cup \cdots \cup G_{i_{m}}$.

If $X$ is a topological space, a neighbourhood of $x \in  X$ is a subset $V$ of $X$ that includes an open set $U$ such that $x \in U.$
\begin{df}
A topological space $X$ is \textbf{locally compact} if every point in $X$ has a compact neighbourhood.
\end{df}

\begin{df}
A subset $S$ of a metric space $X$ is called a \textbf{totally bounded} subset of $X$ if, and only if, for each $r \in \mathbb{R}^{+}$, there is a finite collection of balls of $X$ of radius $r$ that. covers $S$. A metric space $X$ is said to be totally bounded if, and only if, it is a totally bounded subset of itself.
\end{df}


\begin{tm}
\label{tm:comp_closed_totally_bounded}
  A metric space is compact if and only if it is complete and totally bounded.
\end{tm} 
\begin{tm}
\label{tm:complete_closed}
 A subspace $Y$ of a complete metric space is complete if and only if $Y$ is closed.
\end{tm}
\begin{tm}
\label{tm:closed_subset_compact}
Every closed subset of a compact space is compact.
\end{tm}
\begin{df}
A topological space $X$ is said to be \textbf{separable} if it contains a countable dense subset.
\end{df}
\begin{tm}
Every metric subspace of separable metric space is separable.
\end{tm}
\subsection{Measure theory}
We will need following definitions and results from measure theory.
\begin{tm}[\cite{Open_disj}] 
\label{tm:pixel}
Every open subset $U$ of $\mathbb \mathbb R^d,$ $d \geq 1,$ can be written as a countable union of  disjoint half-open cubes of form  $A^{(n)}_{i_1,\ldots,i_d}=\prod_{k=1}^d\left[\frac{i_k}{2^n},\frac{i_k+1}{2^n}\right\rangle,$ $n \in \mathbb N, \ i_1,\ldots, i_d \in \mathbb Z.$
\end{tm}

\begin{df} Let $\mathcal{A}$ be a $\sigma$-algebra on $\mathbb{R}^{d}$ that includes the $\sigma$-algebra $\mathfrak{B}\left(\mathbb{R}^{d}\right)$ of Borel sets. A measure $\mu$ on $\left(\mathbb{R}^{d}, \mathcal{A}\right)$ is \textbf{regular} if
\begin{itemize}
\item[(a)] (locally finite) each compact subset $K$ of $\mathbb{R}^{d}$ satisfies $\mu(K)<+\infty$,
\item[(b)] (outer regular) each set $A$ in $\mathcal{A}$ satisfies
$$
\mu(A)=\inf \{\mu(U): U \text { is open and } A \subseteq U\} \text {, and }
$$
\item[(c)] (inner regular) each open subset $U$ of $\mathbb{R}^{d}$ satisfies
$$
\mu(U)=\sup \{\mu(K): K \text { is compact and } K \subseteq U\} .
$$
\end{itemize}
\end{df}

\begin{tm}[\cite{MT}]
\label{tm:regular}
Any finite measure on $\mathbb R^d$ is regular.
\end{tm}
The  Lebesgue measure on $\mathbb{R}^{d}$ is a regular  measure (see e.g. \cite{MT}). However, not all $\sigma$-finite measures on $\mathbb{R}^{d}$ are regular \cite[Corollary 13.7]{Kl}.
Also, there are some outer regular measures that are not $\sigma$-finite. For example, define $\mu:\mathfrak{B}(\mathbb R)\to [0,\infty]$ by $\mu(A)=\left\{\begin{array}{cc}
\lambda(A),& 0 \notin A,\\
\infty ,& 0 \in A.
\end{array}\right.$
It is easy to see that $\mu$ is a measure on $(\mathbb R,\mathfrak{B}(\mathbb R))$ that is not $\sigma$-finite but is outer regular.
\begin{df}
If $(E, \mathcal{E}, \mu)$ is a measure space, a set $A \in \mathcal{E}$ is called an \textbf{atom} of $\mu$ iff $0<\mu(A)<\infty$ and for every $C \subset A$ with $C \in \mathcal{E}$, either $\mu(C)=0$ or $\mu(C)=\mu(A)$.

A measure without any atoms is called \textbf{non-atomic.}

A measure space $(E, \mathcal{E}, \mu)$, or the measure $\mu$, is called \textbf{purely atomic} if there is a collection $\mathcal{C}$ of atoms of $\mu$ such that for each $A \in \mathcal{E}, \mu(A)$ is the sum of the numbers $\mu(C)$ for all $C \in \mathcal{C}$ such that $\mu(A \cap C)=\mu(C)$.
\end{df}
\begin{lm}[\cite{Chu}]
\label{lm:atom_in_R}
An atom of any finite measure $\mu$ on $(\mathbb{R}^d, \mathcal{B}(\mathbb R^d))$ is a singleton $\{x\}$ such that $\mu(\{x\})>0$.
\end{lm}

\begin{lm}[\cite{Alp}]
\label{atoms_singl}
 Any atom of a Borel measure on a second countable Hausdorff space includes a singleton of positive measure.

In particular, a Borel measure on a second countable Hausdorff space is nonatomic if and only if every singleton has measure zero.
\end{lm}

A measure space $(E, \mathcal{E}, \mu)$ is  localizable if there is a collection $\mathcal{A}$ of disjoint measurable sets of finite measure, whose union is all of $X$, such that for every set $B \subset X, B$ is measurable if and only if $B \cap C \in \mathcal{E}$ for all $C \in \mathcal{A}$, and then $\mu(B)=\sum_{C \in \mathcal{A}} \mu(B \cap C)$. 
Some examples of localisable measures are the $\sigma$-finite ones or counting measures on possibly uncountable sets.
\begin{tm}[\cite{Dud}]
\label{tm: atomic_nonatomic}Let $(E,\mathcal E,\mu)$ be a localisable measure space. Then there exists measures $\nu$ and $\rho$ such that $\mu=\nu+\rho,$ $\nu$ is purely atomic and $\rho$ is non-atomic. 
\end{tm}
\subsection{Boolean algebra}
In this section, we present basics concerning Boolean algebras of sets. For more details, see e.g. \cite{Vl} or \cite{BA2}.

\begin{df}
A \textbf{Boolean algebra (B.A.)} is a structure $(\mathcal X,\cup,\cap,(\cdot)^c,0,1)$ with two binary operations $\cup$ and $\cap,$ a unary operation $(\cdot)^c$ and two distinguished elements $0$ and $1$ such that for all $\mathsf A, \mathsf B$ and $\mathsf C$ in $\mathcal X,$
\[
\begin{tabular}{ll}
$\mathsf A\cup( \mathsf B\cup \mathsf C)=(\mathsf A\cup \mathsf B)\cup \mathsf C,$ &  $\mathsf A\cap(\mathsf B\cap \mathsf C)=(\mathsf A\cap \mathsf B)\cap \mathsf C,$ \\
 $\mathsf A\cup \mathsf B= \mathsf B\cup \mathsf A,$ & $ \mathsf A\cap \mathsf B= \mathsf B \cap \mathsf A,$\\
 $ \mathsf A\cup(\mathsf A\cap \mathsf B)= \mathsf A,$ &  $\mathsf A\cap (\mathsf A\cup \mathsf B)=\mathsf A,$\\
 $ \mathsf A \cap( \mathsf B\cup \mathsf C)=(\mathsf A \cap \mathsf B)\cup(\mathsf A\cap\mathsf  C),$ & $\mathsf A\cup(\mathsf B\cap \mathsf C)=(\mathsf A\cup \mathsf B)\cap(\mathsf A\cup \mathsf C),$\\
$\mathsf A\cup(\mathsf A)^c=1,$ &  $\mathsf A\cap(\mathsf A)^c=0.$
\end{tabular}
\]
\end{df}
\begin{df}
Let $\mathcal X$ be a B.A. 
The B.A. $\mathcal X $ is  \textbf{normed  (N.B.A.)} if  there exists a $\sigma$-additive strictly 
positive finite measure $\mu$ (i.e. $\mu(\mathsf A)=0$ implies $\mathsf A=0$) defined on it. In this case, we use the notation $(\mathcal X,m )$.
\end{df}

On  $\mathcal X\times \mathcal X$ we can define a relation $\subseteq$ by setting
$\mathsf A\subseteq \mathsf B$ if $\mathsf A\cup \mathsf  B= \mathsf B.$
It is easy to verify that $\subseteq$ is a partial order relation.

\begin{df}
B.A. $\mathcal X$ is \textbf{complete} if for every non-empty subset 
$\mathcal C \subseteq \mathcal X$ has its infimum and supremum.
\end{df}

Let $(E, {\mathcal E}, \mu)$ be a finite measure space. 
We can define equivalence relation $\sim$ on $ {\mathcal E }\times \mathcal E$ by setting  $A\sim B$ if and only if $\mu(A \Delta B)=0, $ where $A \Delta  B = ( A\cap  B^c)\cup( A^c \cap B)$ is the symmetric difference between the sets $A$ and $B$ ($A^c$ and $B^c$ denote the complements of $A$ and $B$, respectively). 

Let $[A]=\{ B \in \mathcal E: \mu( A\Delta B)=0\} \in \widetilde{\mathcal E}.$ 
Then $\widetilde{\mathcal E}=\{ [A]:  A \in \mathcal{E}\}$  a quotient space of $\mathcal E$ by $\sim$ is a complete N.B.A. endowed with the measure $\widetilde{\mu}$ defined by $\widetilde{\mu}([A]):=\mu(A)$.

The inverse result also holds. Namely, for each complete N.B.A. $({\mathcal X},m)$, there exists a measure space $(E, {\mathcal E}, \mu)$ such that the N.B.A. $(\mathcal X,m)$ is isomorphic to $(\widetilde{\mathcal E}, \widetilde{\mu})$ (see \cite{Vl}).
Therefore, further on we focus on investigating properties of $(\widetilde{\mathcal E}, \widetilde{\mu}).$ We generalise above setting, by letting measure $\mu$ be an arbitrary, possibly non-finite.


Define $d_{\mu}: \widetilde{\mathcal E}\times \widetilde{\mathcal E}\to \left[0,\infty\right]$ by
$$d_{\mu}([A],[B]):=\left(\widetilde{\mu}([A]\Delta[B])\right)^{(1/2)}=\left(\mu(A\Delta B)\right)^{(1/2)}.$$
It is easy to see that $d_{\mu}$ is a metric on $\widetilde{\mathcal{E}}$ possibly taking infinite values.
We suppose the topology on $\widetilde{\mathcal{E}}$ is generated by $d_{\mu}.$ We are interested in topological properties of  $(\widetilde{\mathcal{E}},d_{\mu}).$
\begin{re}
Let us mention that there are many topologies introduced in B.A.s. The most popular among them is the order topology.
It is known that the topology of the metric space $(\widetilde{\mathcal{E}},d_{\mu})$ coincides with the ordered topology (see \cite{Vl}).
\end{re}

Denote $L^2(E,\mathcal{E},\mu)$ a Hilbert space of measurable functions which are square integrable with respect to the measure $\mu$, where functions which agree $\mu$ almost everywhere are identified.
Let $\mathcal I=\{\1_A, A \in \mathcal E\}=\{\1_A: \mu(A)<\infty \}\subset L^2(E,\mathcal{E},\mu)$ where $$\1_A(x)=\left\{ \begin{array}{cc} 0, & x \notin A,\\ 1,& x \in A, \end{array} \right.$$ stands for \textbf{indicator} of set $A$ or a \textbf{characteristic function} of set $A.$

If $\mu(E)<\infty,$ we can define $\iota :\widetilde{\mathcal E} \to \mathcal I$ by $\iota([A])=\1_A.$ Since
$$d_\mu([A],[B]) =\left(\int_{E} \left|\1_{A} - \1_{B} \right|^2 d\mu\right)^{(1/2)},$$ $\mathcal E$ and $\mathcal I$ are isometric.

Suppose that $1_{A_n}$ converges to $f$ in $L^2(E,\mathcal{E},\mu).$ Since, $L^2(E,\mathcal{E},\mu)$ is complete, $f \in L^2(E,\mathcal{E},\mu).$ Let us show that $f$ is an indicator function of some measurable set.   
There exists a subsequence $(\1_{A_{n_k}})_k$ such that $\lim\limits_k \1_{A_{n_k}}(x)=f(x), \ \mu-a.e. $ (see e.g. \cite[Theorem 16.25]{Y})
Also, following
\begin{align*}
f(x)&=\liminf_k\1_{A_{n_k}}(x)=\1_{\liminf_k A_{n_k}}(x)\leq \1_{\limsup_k A_{n_k}}(x)\\
&=\limsup \1_{A_{n_k}}(x)=f(x), \mu-a.e.,
\end{align*}
 we conclude that $\mathcal I$ is closed. 
 Following Theorem \ref{tm:complete_closed} we can conclude that $\mathcal I$ is complete metric subspace of $L^2(E,\mathcal{E},\mu).$ Therefore, we have shown that in case $\mu$ is finite  $(\widetilde{\mathcal E},d_{\mu})$ is complete metric space.
 Let us show that this holds in a general case when $\mu$ is not finite.
\begin{tm}
\label{tm:compl}
$(\widetilde{\mathcal E},d_{\mu})$ is complete metric space.
\end{tm}
\begin{proof}
Suppose that $([A_n])_{n \in \mathbb N}$ is a Cauchy sequence in $\widetilde{\mathcal E}.$ Then for fixed $\epsilon >0$ there exists $n_{\epsilon} \in \mathbb N$ such that for $n,m\geq n_{\epsilon}$ $d_{\mu}([A_n],[A_m])<\epsilon.$ We define
$f_{n}= \1_{A_n}-\1_{A_{n_\epsilon}}, n\geq n_{\epsilon}.$ Since $\int |f_n| d\mu=\mu(A_n \Delta A_{n_{\epsilon}})=d_{\mu}([A_n]\Delta [A_{n \epsilon}])\leq \epsilon, $ so $f_n \in L^p(E,\mathcal E,\mu).$ Also, $(f_n)$ is a Cauchy sequence, and since $L^1(E,\mathcal E,\mu)$ is complete, there exists $f \in L^1$ such that $\lim\limits_{n}\int_E|f_n-f|^pd\mu=0.$ Furthermore, there exists a subsequence $(f_{n_k})_k$ such that  $\lim_k f_n{_k}(x)=f(x), \ \mu-a.e.$  
It holds
\begin{align*}
f(x)&=\liminf_k (\1_{A_{n_k}}(x)-\1_{A_{n_{\epsilon}}})=\1_{\liminf_k A_{n_k}}(x) -\1_{A_{n_{\epsilon}}}\leq \1_{\limsup_k A_{n_k}}(x)-\1_{A_{n_{\epsilon}}}\\
&=\limsup_k (\1_{A_{n_k}}(x)-\1_{A_{n_{\epsilon}}})=f(x), \mu-a.e.,
\end{align*}
which shows that $\1_{{\liminf_k A_{n_k}}}=\1_{\limsup_k A_{n_k}} \mu-a.e,$ or equivalently $[{\liminf_k A_{n_k}}]=[{\limsup_k A_{n_k}}].$ Following
$d_{\mu}([A_n],[{\liminf_k A_{n_k}}])=\int_E|\1_{A_n}-\1_{{\liminf_k A_{n_k}}}|d\mu=\int_E|\1_{A_n}-\1_{A_{n_{\epsilon}}}+\1_{A_{n_{\epsilon}}}-\1_{{\liminf_k A_{n_k}}}|d\mu=\int_E|f_n-f|d\mu=0,$ the sequence $([A_n])_n$ is convergent, so $(\widetilde{\mathcal E},d_{\mu})$ is complete. 
\end{proof}
\begin{re}
In case $\mu(E)=\infty,$ one can also consider $\widetilde{\mathcal E}_{fin}=\{ [A]: A\in \mathcal E, \mu(A)<\infty\}.$ It is easy to see that $(\widetilde{\mathcal E}_{fin},d_{\mu})$ is isometric to $\mathcal I \subset L^2(E,\mathcal E,\mu),$ so it is complete metric subset of $(\widetilde{\mathcal E},d_{\mu}).$ However, $\widetilde{\mathcal E}_{fin}$ is not a B.A. since it is not e.g. closed under complements.
\end{re}
\section{Main result}\label{sec1} 
\setcounter{equation}{0}
Before we show the main result, in order to get intuition,  we first start with a motivating example.

Suppose that $K=[0,1]\times [0,1]$ (an observation window) and consider $(E,\mathcal E,\mu)=(K,\mathfrak{B}(K),\lambda|_{\mathfrak{B}(K)}), $ where $\mathfrak{B}(K)$ is Borel $\sigma$-algebra on $K$ and $\lambda|_{\mathfrak{B}(K)}$ Lebesgue measure.

If we consider a ball in $(\widetilde{\mathcal E},d_{\mu})=(\widetilde{\mathfrak B(K)},d_{\lambda})$ of radius $\epsilon>0$ it holds:
\begin{align*}
    B([A],\epsilon)&=\{[B] \in \widetilde{\mathfrak B(K)}: d_{\lambda}([A],[B])<\epsilon\}\\
    &=\{ [B] \in \widetilde{\mathfrak{B}(K)},  \lambda(A\Delta B)<\epsilon^2\}.
\end{align*}

For each $ n \in \mathbb N,$ we can partition $K$ into $2^{2n}$ smaller squares $A^{(n)}_{(i,j)}=[(i-1)/2^n,i/2^n]\times[(j-1)/2^n,j/2^n], \ i,j=1,\ldots,2^n.$ Intuitively, we pixelise the unit square by a $2^n\times 2^n$ net.

We show that for each $\epsilon$ we can pixelise the unit square fine enough so that the error of approximation of the set $B$ would be less than $\epsilon.$
Denote by $I_n=\{1,2,\ldots,2^n\}\times\{1,2,\ldots,2^n\}.$
\begin{lm} \label{lm:approx} For an arbitrary $B \in \mathfrak{B}(K)$ and an arbitrary $\epsilon >0,$ there exists $n \in \mathbb N$ and $I \subset I_n $ such that
\[ \lambda\left(\bigcup\limits_{(i,j)\in I}A^{(n)}_{(i,j)}\Delta B\right)<\epsilon.
\]
\end{lm}
This result follows directly from Lemma \ref{lm:aprox} which we prove later in the paper.

Note that the family  $\{[A^{(n)}_{(i,j)}]: i,j,n\in \mathbb N, i,j\leq 2^n\}$ is countable dense subset of $(\widetilde{\mathfrak B(K)},d_{\lambda}),$ so $(\widetilde{\mathfrak B(K)},d_{\lambda})$ is separable.

Following Lemma \ref{lm:approx}, for arbitrary $\epsilon>0$ the collection of balls 
\begin{equation}
\label{eq:cover}
\mathcal C_{\epsilon}=\{ B([A],\epsilon): \text{ $ \exists n \in \mathbb N$ such
that $A=\bigcup\limits_{(i,j)\in I}A^{(n)}_{(i,j)}$ for some $I \subset I_n.$}\}
\end{equation}
is an infinite (countable) open cover of $(\widetilde{\mathfrak B(K)},d_{\lambda}).$ (Since for every $\epsilon>0$ and arbitrary $B \in \mathfrak B(K)$ there exists $A$ in form $A=\bigcup\limits_{(i,j)\in I}A^{(n)}_{(i,j)}$ for some $I \subset I_n$ such that $[B] \in B([A],\epsilon)$ )

Suppose that $(\widetilde{\mathfrak B(K)},d_{\lambda})$ is compact, therefore the open cover (\ref{eq:cover}) should have a finite subcover. It means that there exists $m \in \mathbb N$ such that the collection of open balls 
\[
\mathcal C_{\epsilon,fin}^{m} = \{ B([A],\epsilon): \text{ $A=\bigcup\limits_{(i,j)\in I}A^{(m)}_{(i,j)}$ for some $I \subset I_m.$}\}
\]
covers $(\widetilde{\mathfrak B(K)},d_{\lambda}).$

However, if we take $\epsilon =\frac{1}{\sqrt{2}}$ and define a set \begin{equation}
\label{eq:Tm}
    T_m=\bigcup\limits_{i,j=1}^{m}T_{(i,j)},
\end{equation}  where $T_{(i,j)}=\{ (x,y)\in K: x\in[(i-1)/2^m,i/2^m],y\in [(j-1)/2^m,j/2^m], y\leq x -i/2^m+j/2^m \}$ it is easy to see (Left plot in Figure \ref{fig1} provides a visualisation of the set $T_2$)
\begin{figure}[H]
    \centering
\begin{tikzpicture}
\filldraw[color=gray!45, fill=gray!45, very thick] (0,0) -- (1,0) -- (1,1) -- cycle;
\filldraw[color=gray!45, fill=gray!45, very thick] (1,0) -- (2,0) -- (2,1) -- cycle;
\filldraw[color=gray!45, fill=gray!45, very thick] (2,0) -- (3,0) -- (3,1) -- cycle;
\filldraw[color=gray!45, fill=gray!45, very thick] (3,0) -- (4,0) -- (4,1) -- cycle;
\filldraw[color=gray!45, fill=gray!45, very thick] (0,1) -- (1,1) -- (1,2) -- cycle;
\filldraw[color=gray!45, fill=gray!45, very thick] (1,1) -- (2,1) -- (2,2) -- cycle;
\filldraw[color=gray!45, fill=gray!45, very thick] (2,1) -- (3,1) -- (3,2) -- cycle;
\filldraw[color=gray!45, fill=gray!45, very thick] (3,1) -- (4,1) -- (4,2) -- cycle;
\filldraw[color=gray!45, fill=gray!45, very thick] (0,2) -- (1,2) -- (1,3) -- cycle;
\filldraw[color=gray!45, fill=gray!45, very thick] (1,2) -- (2,2) -- (2,3) -- cycle;
\filldraw[color=gray!45, fill=gray!45, very thick] (2,2) -- (3,2) -- (3,3) -- cycle;
\filldraw[color=gray!45, fill=gray!45, very thick] (3,2) -- (4,2) -- (4,3) -- cycle;
\filldraw[color=gray!45, fill=gray!45, very thick] (0,3) -- (1,3) -- (1,4) -- cycle;
\filldraw[color=gray!45, fill=gray!45, very thick] (1,3) -- (2,3) -- (2,4) -- cycle;
\filldraw[color=gray!45, fill=gray!45, very thick] (2,3) -- (3,3) -- (3,4) -- cycle;
\filldraw[color=gray!45, fill=gray!45, very thick] (3,3) -- (4,3) -- (4,4) -- cycle;
\draw[black, very thick] (0,0) rectangle (4,4);
\end{tikzpicture}
\hspace{1 cm}
  \begin{tikzpicture}
\filldraw[color=gray!95, fill=gray!95, very thick] (0.5,0) -- (1,0) -- (1,0.5) -- cycle;
\filldraw[color=gray!95, fill=gray!95, very thick] (1.5,0) -- (2,0) -- (2,0.5) -- cycle;
\filldraw[color=gray!95, fill=gray!95, very thick] (2.5,0) -- (3,0) -- (3,0.5) -- cycle;
\filldraw[color=gray!95, fill=gray!95, very thick] (3.5,0) -- (4,0) -- (4,0.5) -- cycle;
\filldraw[color=gray!95, fill=gray!95, very thick] (0.5,1) -- (1,1) -- (1,1.5) -- cycle;
\filldraw[color=gray!95, fill=gray!95, very thick] (1.5,1) -- (2,1) -- (2,1.5) -- cycle;
\filldraw[color=gray!95, fill=gray!95, very thick] (2.5,1) -- (3,1) -- (3,1.5) -- cycle;
\filldraw[color=gray!95, fill=gray!95, very thick] (3.5,1) -- (4,1) -- (4,1.5) -- cycle;
\filldraw[color=gray!95, fill=gray!95, very thick] (0.5,2) -- (1,2) -- (1,2.5) -- cycle;
\filldraw[color=gray!95, fill=gray!95, very thick] (1.5,2) -- (2,2) -- (2,2.5) -- cycle;
\filldraw[color=gray!95, fill=gray!95, very thick] (2.5,2) -- (3,2) -- (3,2.5) -- cycle;
\filldraw[color=gray!95, fill=gray!95, very thick] (3.5,2) -- (4,2) -- (4,2.5) -- cycle;
\filldraw[color=gray!95, fill=gray!95, very thick] (0.5,3) -- (1,3) -- (1,3.5) -- cycle;
\filldraw[color=gray!95, fill=gray!95, very thick] (1.5,3) -- (2,3) -- (2,3.5) -- cycle;
\filldraw[color=gray!95, fill=gray!95, very thick] (2.5,3) -- (3,3) -- (3,3.5) -- cycle;
\filldraw[color=gray!95, fill=gray!95, very thick] (3.5,3) -- (4,3) -- (4,3.5) -- cycle;
\draw[black, very thick] (0,0) rectangle (4,4);
\draw[black, very thick] (0,1) -- (4,1);
\draw[black, very thick] (0,2) -- (4,2);
\draw[black, very thick] (0,3) -- (4,3);
\draw[black, very thick] (1,0) -- (1,4);
\draw[black, very thick] (2,0) -- (2,4);
\draw[black, very thick] (3,0) -- (3,4);
\end{tikzpicture}
    \caption{Left plot: Visualisation of set $T_{2}^{\frac{1}{2}}$ (coloured in grey) defined by (\ref{eq:Tm}). Right plot: Visualisation of the set $T_2^{\epsilon}$ for $\epsilon=\frac{1}{8}$ (coloured in dark grey) defined by (\ref{eq:Tm})}
    \label{fig1}
\end{figure}
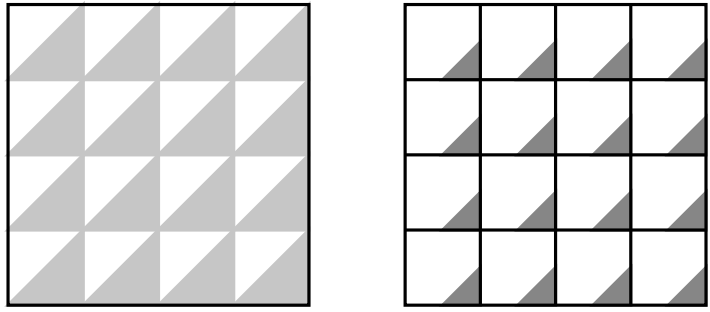

that $\lambda(A\Delta T_m)=1/2,$ for each $A$  such
that $A=\bigcup\limits_{(i,j)\in I}A^{(n)}_{(i,j)}$ for some $I \subset I_n, n\leq m.$  
Therefore $[T_m]$ is not contained in any ball in $\mathcal C_{\frac{1}{\sqrt{2}},fin}^{m}$ so  $\mathcal C_{\frac{1}{\sqrt{2}},fin}^{m}$ cannot be a cover of $(\widetilde{\mathfrak B(K)},d_{\lambda}).$ So countable open cover $\mathcal{C}_{\frac{1}{\sqrt{2}}}$ of $(\widetilde{\mathfrak B(K)},d_{\lambda})$ has no finite subcover.  
We can conclude that $(\widetilde{\mathfrak B(K)},d_{\lambda})$ is not compact.

In order to show that $(\widetilde{\mathfrak B(K)},d_{\lambda})$ is not locally compact, we will prove that  closed ball with the centre in $[\emptyset]$ and radius $\epsilon<\frac{1}{\sqrt{2}}$ denoted by $  \overline{B}([A],\epsilon)=\{ [B] \in \widetilde{\mathfrak B(K)} : \lambda(A\Delta B)\leq \epsilon^2\}$ is not compact.

Since $\mathcal C_{\epsilon}$ covers $\widetilde{\mathfrak B(K)}$ it also covers $\overline{B}\left([\emptyset],\epsilon\right)$ 
We will show that open cover $\mathcal C_{\epsilon}$ cannot be reduced to a finite subcover.
For that purpose, for $0<\epsilon^2\leq\frac{1}{2}$ we define a set \begin{equation}
\label{eq:Tm}
    T_{m}^{(\epsilon)}=\bigcup\limits_{(i,j)\in I_m}T_{(i,j)}^{(\epsilon)},
\end{equation}  where $T_{(i,j)}^{(\epsilon)}=\{ (x,y)\in K: x\in[(i-1)/2^m+(1-\epsilon^2/2)/2^m,i/2^m],y\in [(j-1)/2^m,j/2^m], y\leq x -i/2^m+j/2^m \}$ 
of $2^{2m}$ disjoint triangles whose union has Lebesgue measure equal to $\epsilon^2,$ so $[{T_m}] \in \overline{B}\left([\emptyset],\epsilon\right)$ for each $m \in \mathbb N.$
Right plot in Figure \ref{fig1} provides a visualisation of the set $T_2^{\epsilon}$ for $\epsilon=\frac{1}{8}.$
For each $A$ such
that $A=\bigcup\limits_{(i,j)\in I}A^{(m)}_{(i,j)}$ for some $I \subset I_m,$ it holds 
\begin{align}
\label{eq:Tm_epsilon_diff}
\mu(A\Delta T_m^{(\epsilon)})&= \mu(A\diagdown T^{(\epsilon)}_m)+\mu(T^{(\epsilon)}_m\diagdown A)\\ \nonumber
&= \sum\limits_{(i,j)\in I}\mu(A_{(i,j)}\diagdown T^{(\epsilon)}_{m})+\sum\limits_{(i,j)\in I_m\diagdown I}\mu( T^{(\epsilon)}_{m}\diagdown A_{(i,j)})\\ \nonumber
&=(1-\frac{\epsilon^2}{2^{2m}})|I|+\frac{\epsilon^2}{2^{2m}}(|I_m|-|I|)\\
&\geq \epsilon^2, \nonumber
\end{align} 
where the last inequality follows from the fact that from $0< \epsilon^2 \leq \frac{1}{2}$ follows $1-\epsilon^2>\epsilon^2.$  
So $[T_m^{(\epsilon)}] \in \overline{B}\left([\emptyset],\epsilon\right)$ but it is not in any ball in $\mathcal C_{\epsilon,fin}^{m}.$ We conclude   $\mathcal C_{\epsilon,fin}^{m}$ cannot be a cover of $\overline{B}\left([\emptyset],\epsilon\right).$  Since countable open cover $\mathcal{C}_{\epsilon}$ of $\overline{B}\left([\emptyset],\epsilon\right)$ has no finite subcover, $\overline{B}\left([\emptyset],\epsilon\right)$ is not compact. 
We conclude that $[\emptyset]$ does not have a compact neighbourhood, so $\widetilde{\mathfrak B(K)}$ is not locally compact.

In order to generalise these ideas on arbitrary  measure space  $(E,\mathcal E,\mu)$ we introduce following definitions.

\begin{df}
Let $(E,\mathcal E,\mu)$ be a measure space and $\mathcal F=(\mathcal F_n)_{n \in \mathbb N}$ a filtration, i.e. $\mathcal F_n \subseteq  \mathcal E $ is $\sigma$-algebra such that $\mathcal F_n \subseteq \mathcal F_{n+1}.$ Let $\mathcal E'$ be an arbitrary subset of $\mathcal E.$

Measure $\mu$ is \textbf{approximable} on $\mathcal E'$ with respect to $\mathcal F$ if for each  $A \in \mathcal E'$ and each $\epsilon>0$ there exists $n_{\epsilon}\in \mathbb N$ and $A'\in \mathcal{F}_{n_{\epsilon}}$ such that
\[
\mu(A\Delta A')<\epsilon.
\]

Measure $\mu$ is \textbf{uniformly approximable} on $\mathcal E'$ with respect to $\mathcal F$ if for each $\epsilon>0$ there exists $n_{\epsilon}\in \mathbb N$ such that for every $A \in \mathcal{E'}$ there exists $A'\in \mathcal{F}_{n_{\epsilon}}$ so that
\[
\mu(A\Delta A')<\epsilon.
\]

\end{df}

For arbitrary $\mathcal E' \subseteq \mathcal E$ denote by $\widetilde{\mathcal E'}=\{ [A]: A\in \mathcal E'\}.$
\begin{tm} \label{tm:sep} $\widetilde{\mathcal E'}$ is separable in $(\widetilde{\mathcal E},d_\mu)$ if and only if there exist $\mathcal F=(\mathcal F_n)_{n \in \mathbb N}$ a filtration for which $|\mathcal F_n|$ is finite for each $n \in \mathbb N$ such that $\mu$ is approximable on $\mathcal E'$ with respect to $\mathcal F$.
\end{tm}
\begin{proof}
Suppose that $\mu$ is approximable on $\mathcal E'$ wih respect to $\mathcal F.$ Denote by $\widetilde{\mathcal E'}_c=\{ [B]: B \in \cup_{n \in \mathbb N }\mathcal F_n\}.$ Since  $\cup_{n \in \mathbb N}\mathcal F_n$ is countable, $\widetilde{\mathcal E'}_c$ is also countable. Let us show that $\widetilde{\mathcal E'}_c$ is dense in $\widetilde{\mathcal E'}$. For arbitrary $\epsilon >0$ and arbitrary $A \in \mathcal E'$ there exists $A' \in \cup_{n \in \mathbb N}\mathcal F_n$ such that $\mu(A\Delta A')<\epsilon^2,$ so $[A] \in B([A'],\epsilon),$ and therefore $\widetilde{\mathcal E'}_c$ is countable dense subset of $\widetilde{\mathcal E'}$ and $\widetilde{\mathcal E'}$ is separable.

If $\widetilde{\mathcal E'}$ is separable, then there exists a countable dense subset of $\widetilde{\mathcal E'},$ denote it by $\widetilde{\mathcal E'}_c.$ W can represent $\widetilde{\mathcal E'}_c$
as $\{ [B]: B\in \mathcal B\}$ for some countable $\mathcal B=\{ B_n: n \in \mathbb N\} \subset \mathcal E'.$  Take $\mathcal F_n=\sigma (B_1\ldots B_n).$ Then for each $\epsilon>0$ and each $A \in \mathcal E'$ there exists $B_n \in \mathcal B$ such that $[A] \in B([B_n],\epsilon)$ so that $\mu(A\Delta B_n)<\epsilon^2.$
Therefore, $\mu$ is aproximable on $\mathcal E'$ with respect to $( F_n)_{n\in \mathbb N}.$
\end{proof}
\begin{tm} \label{tm:comp} $\widetilde{\mathcal E'}$  is totally bounded in $(\widetilde{\mathcal E},d_\mu)$ if and only if there exist exists a filtration $\mathcal F=(\mathcal F_n)_{n \in \mathbb N}$  such that $\mu$ is uniformly approximable on $\mathcal E'$ with respect to $\mathcal F$ and $|\mathcal F_n|$ is finite for each $n \in \mathbb N.$
\end{tm}
\begin{proof}
Suppose that $\widetilde{\mathcal E'}$ is totally bounded. Then for each $\epsilon>0$ there exists a finite family of sets $\mathcal A_\epsilon=\{A^{(\epsilon)}_1,\ldots, A^{(\epsilon)}_{m_{\epsilon}}\}$ such that $\widetilde{\mathcal E'} \subseteq  \cup_{i=1}^{m} B([A^{(\epsilon)}_i],\epsilon).$ For arbitrary $B \in \mathcal E',$ $[B] \in \cup_{i=1}^{n} B([A^{(\epsilon)}_i],\epsilon),$ so there exists $A^{(\epsilon)}_i$ such that $\mu(A^{(\epsilon)}_i\Delta B)<\epsilon^2.$ 
If we consider $\epsilon_n=\frac{1}{n}$ and take $\mathcal{F}_n=\sigma\left(\cup_{i=1}^{n}\mathcal A_{\frac{1}{n}}\right)$ we get that $\mu$ is uniformly approximable on $\mathcal E'$ with respect to $(\mathcal F_n).$

Conversely, if there exists $(\mathcal F_n)$ where $|\mathcal F_n|$ is finite for each $n \in \mathbb N$ and $\mu$ is uniformly approximable on $\mathcal E'$ with respect to $(\mathcal F_n).$ For each $\epsilon>0$ there exists $n$ such that for each $B \in \mathcal E'$ there exists $A_n \in \mathcal F_n,$  $\mu(A_n\Delta B)<\epsilon^2.$ So 
$[B] \in B([A_n],\epsilon)$ and therefore
$\widetilde{\mathcal E'} \subseteq \cup_{A \in \mathcal{F}_n}B([A],\epsilon),$ which shows that $\widetilde{\mathcal E'}$ is totally bounded. 
\end{proof}
\begin{co}
\label{co:un_approx_whole_E}
\begin{itemize}
    \item[(a)] $(\widetilde{\mathcal E},d_{\mu})$ is is separable if and only if there exist $\mathcal F=(\mathcal F_n)_{n \in \mathbb N}$ a filtration for which $|\mathcal F_n|$ is finite for each $n \in \mathbb N$ such that $\mu$ is approximable on $\mathcal E$ with respect to $\mathcal F$.
    \item[(b)] $(\widetilde{\mathcal E},d_{\mu})$ is is compact if and only if there exist $\mathcal F=(\mathcal F_n)_{n \in \mathbb N}$ a filtration for which $|\mathcal F_n|$ is finite for each $n \in \mathbb N$ such that $\mu$ is approximable on $\mathcal E$ with respect to $\mathcal F$.
\end{itemize}
\end{co}
\begin{proof}
The $(a)$ part follows directly from Theorem \ref{tm:sep}. The $(b)$ part follows from  Theorem \ref{tm:comp_closed_totally_bounded}, Theorem \ref{tm:compl} and Theorem \ref{tm:comp}. 
\end{proof}
Intuitively speaking, we can imagine a finite filtration $(\mathcal F_n)$   as a way to pixelise $E$  that in each step (as $n$ grows) we get a finer "grid". Following Theorem \ref{tm:sep} and Theorem \ref{tm:comp}, $(\widetilde{\mathcal E},d_{\mu})$ is separable if for each measurable set we can find a level of pixelisation such that the error is smaller than arbitrary $\epsilon>0$ and $(\widetilde{\mathcal E},d_{\mu})$ is compact if for each $\epsilon>0$ we can find a level of pixelisation such that all measurable sets are well approximated on this level, i.e. the error of pixelisation is smaller than $\epsilon$ for each measurable set.

Let us now classify  measures $\mu$ on $(\mathbb R^d,\mathfrak B (\mathbb R^d))$ based on topological properties of corresponding $(\widetilde{\mathfrak B (\mathbb R^d)},d_{\mu}).$ 

Further on, denote by $A^{(n)}_{(i_1,\ldots,i_d)}=\prod\limits_{j=1}^d\left[(i_j-1)/2^n,i_j/2^n\right\rangle, \ i_1,\ldots,i_d \in \mathbb Z,$ a $d$-dimensional half-open interval in $\mathbb R^d.$
We prove that an arbitrary Borel set in $\mathbb R^d$ can be approximated by the finite union of disjoint half-open $d$-intervals in a sense that the measure of symmetric difference between the Borel set and the union is arbitrary small.
\begin{lm}
\label{lm:aprox}
If $\mu$ is a outer regular measure on $\mathbb R^d$ then for an arbitrary $B \in \mathfrak{B}(\mathbb R^d)$ such that $\mu(B)<\infty$ and  an arbitrary $\epsilon >0,$ there exists $n_0 \in \mathbb N$ and finite $I\subseteq \mathbb \{-n_02^{n_0}+1\ldots,0, \ldots,n_02^{n_0}\}^d$ such that
\[ \mu \left(\bigcup\limits_{(i_1,\ldots,i_d)\in I}A^{(n_0)}_{(i_1,\ldots,i_d)}\Delta B\right)<\epsilon.
\] 
\end{lm}
\begin{proof}
Let us take an arbitrary $B \in \mathfrak{B}(\mathbb R^d),$ $\mu(B)<\infty$ and an arbitrary $\epsilon>0.$

Space $\mathbb R^d$ can be represented as a decreasing union of the half-open $d$-intervals $\left[-n,n\right\rangle^d, n \in \mathbb N.$
It holds that 
\begin{align*}
\mu(B)&=\mu\left(B\cap \mathbb R^d \right)=\mu\left(B \cap  \cup_{n \in \mathbb N} \left[-n,n\right\rangle^d \right) =\mu\left(  \cup_{n \in \mathbb N} B \cap\left[-n,n\right\rangle^d \right)\\
&=\lim \limits_{n \rightarrow \infty}\mu\left(B \cap\left[-n,n\right\rangle^d\right).
\end{align*}
The last equality follows from continuity of the measure $\mu$ from below with respect to the increasing sequence $\{B \cap\left[-n,n\right\rangle^d\}_n.$

Since $\mu (B)<\infty$ for given $\epsilon$ there exists $n_1 \in \mathbb N $ such that for each $n\geq n_1$ 
\begin{equation}
    \mu\left(B\Delta \left[ -n,n\right\rangle^d\right)=\mu(B)-\mu\left(B\cap \left[ -n,n\right\rangle^d\right)\leq \frac{\epsilon}{3}.
\end{equation}

Since $\mu$ is outer regular, for $B'=B\cap\left[ -n_1,n_1\right\rangle$ and arbitrary $\epsilon>0$ there exist an open set $O$ such that $B \subseteq O$ and

\[
\mu(B'\Delta O)=\mu(O\diagdown B')<\frac{\epsilon}{3}.
\]
Following Theorem \ref{tm:pixel},  $O$ can be represented as 
countable union of almost disjoint half-open cubes $(A_k)_{k\in \mathbb N}.$ Since
\[
\mu(O)=\sum\limits_{k\in \mathbb N}\mu(A_k)<\infty,
\]
for chosen $\epsilon$ there exists $N \in \mathbb N$ such that
\[
\mu\left( \cup_{k=1}^{N} A_k\right)= \sum\limits_{k=1}^{N}\mu(A_k) \geq \mu(O)-\frac{\epsilon}{3},
\]
so since $\cup_{k=1}^{N} A_k \subseteq O$
\[ \mu(\cup_{k=1}^{N} A_k \Delta O)\leq \frac{\epsilon}{3}.
\]
Set $n_2=\max\{n \in \mathbb N: A_k=A^{(n)}_{i_1,\ldots,i_d}, k=1,\ldots,N, (i_1,\ldots,i_d) \in \mathbb Z^d \}.$  
Note that if $m_1<m_2,$ each $A^{(m1)}_{(j_1,\ldots,j_d)}$ can be represented as finite union of disjointed $d$-intervals $A^{(m_2)}_{(j'_1,\ldots,j'_d)}.$

We take $n_0=\max\{n_1,n_2\}$
and define $I=\{(i_1,\ldots,i_d)\in \{-n_02^{n_0},\ldots,n_02^{n_0}\}^d: A^{(n_0)}_{(i_1,\ldots,i_d)}\subset A_k, k=1,\ldots N\}.$

Note that $\cup_{k=1}^{N}=\cup_{(i_1,\ldots,i_d)\in I}A^{(n_0)}_{(i_1,\ldots,i_d)}$

Therefore,
\[ \mu(\cup_{(i_1,\ldots,i_d)\in I}A^{(n_0)}_{(i_1,\ldots,i_d)}\Delta B)\leq \mu(\cup_{k=1}^{N} A_k \Delta O)+\mu(O\Delta B')+\mu(B\Delta B)<\epsilon.
\]
\end{proof}
\begin{tm}
\label{tm:sep_1}
 Let $(E,\mathcal E,\mu)=(\mathbb R^d, \mathfrak{B}(\mathbb R^d),\mu)$ where $\mu$ is an outer regular measure. Then $(\widetilde{\mathcal E}_{fin},d_{\mu})$ is separable.
\end{tm}
\begin{proof}
\begin{figure}[H]
    \centering
    \begin{tikzpicture}
          \foreach \i in {-23,...,24}
    \foreach \j in {-23,...,24} 
    {\draw [draw=white!90!blue] (\i*0.125-0.125,\j*0.125-0.125) rectangle (\i*0.125,\j*0.125);}
    \foreach \i in {-7,...,8}
    \foreach \j in {-7,...,8} 
    {\draw[draw=blue] (\i*0.25-0.25,\j*0.25-0.25) rectangle (\i*0.25,\j*0.25);}
        \foreach \i in {-1,...,2}
    \foreach \j in {-1,...,2} 
    {\draw [draw=black!80!blue, line width=0.5mm] (\i*0.5-0.5,\j*0.5-0.5) rectangle (\i*0.5,\j*0.5);}
    \draw [-stealth](-3.5,0) -- (3.5,0);
     \draw [-stealth](0,-3.5) -- (0,3.5);
    \end{tikzpicture}
    \caption{Visualisation of sets in family $\{ A^{(n)}_{i_1,\ldots,i_d}: (i_1,\ldots, i_d)\in \{-n2^n+1,n2^n\}\}$ for $n=1$ (black), $n=2$ (blue) and $n=3$ (light blue).}
    \label{fig:tm33}
\end{figure}
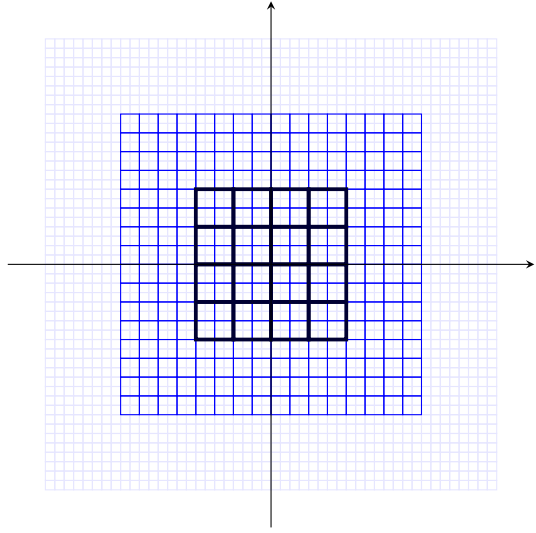
  For each $n \in \mathbb N$ the family $\mathcal{A}_n=\{ A^{(n)}_{i_1,\ldots,i_d}: (i_1,\ldots, i_d)\in \{-n2^n+1,n2^n\}\} \cup \mathbb R^d\diagdown \left[-n,n\right\rangle^d$ is finite (see Figure \ref{fig:tm33} for visualisation). 
  Denote by $$\mathcal F_n=\{ A \in \mathcal{B}(\mathbb R^d): A \text{ can be written as a union of sets from \ } \mathcal{A}_n\} \cup \{\emptyset\}.$$
  Note that since $\mathcal A_n$ forms a finite partition of $\mathbb R^d,$ $\mathcal F_n=\sigma(\mathcal A_n).$
It holds that $|\mathcal F_n|<\infty.$ Also note that $\mathcal F_n \subset \mathcal F_{n+1}.$ Let $\mathcal F=\cup_{n \in \mathbb N}\mathcal F_n,$ and note that  $\mathcal F$ is countable.  
Note that from Lemma \ref{lm:aprox} it follows that the $\mu$ is approximable on $\mathcal E_{fin}$ with respect to $\mathcal F_n),$ and from Theorem \ref{tm:sep} it follows that $(\widetilde{\mathcal E}_{fin},d_{\mu})$ is separable. 
\end{proof}

 \begin{co}
For any outer regular Borel measure $\mu$ on $(\mathbb R^d, \mathfrak{B}( \mathbb R^d)),$ $(\widetilde{\mathfrak{B}( \mathbb R^d)}_{fin},d_{\mu})$ is a Polish space (i.e. complete separable metric space).
\end{co}
\begin{re}
 Since $(\widetilde{\mathfrak{B}(\mathbb R^d)}_{fin},d_{\mu})$ is isometric to set of indicators in $L^2(\mathbb R^d, \mathfrak{B}(\mathbb R^d),\mu),$ we can also conclude that in case of outer regular measure $\mu$ set of indicators in $L^2(\mathbb R^d, \mathfrak{B}(\mathbb R^d),\mu)$ is a Polish space. 
\end{re}
\begin{co}
\label{co:fin_polish}
For any finite Borel measure $\mu$ on $(\mathbb R^d, \mathfrak{B}( \mathbb R^d)),$ $(\widetilde{\mathfrak{B}( \mathbb R^d)},d_{\mu})$ is a Polish space (i.e. complete separable metric space).
\end{co}
 Note that Corollary \ref{co:fin_polish} could be proven using separability of set indicators in $L^2(\mathbb R^d, \mathfrak{B}(\mathbb R^d),\mu).$

As it has been already mentioned in Introduction,  if $\mathcal E$ is countably generated and $\mu$ is a $\sigma$-finite measure then $L^2(E,\mathcal E,\mu)$ and set of indicators in $L^2(E,\mathcal E,\mu)$ are separable. For a finite $\mu,$ since $(\widetilde{\mathcal E}_{fin},d_{\mu})=(\widetilde{\mathcal E},d_{\mu})$ and $(\widetilde{\mathcal E}_{fin},d_{\mu})$ is homeomorfic to set of indicators, we can conclude $(\widetilde{\mathcal E},d_{\mu})$ is separable. 
We provide an alternative proof of this fact using the notion of approximability.
\begin{tm}
\label{tm:sep_2}
Suppose that there exists $\mathcal C$ a countable family of subsets of $E$  such that $\mathcal E=\sigma(\mathcal C)$ and $(E,\mathcal E,\mu)$ is a finite measure space. Then $(\widetilde{\mathcal E},d_{\mu})$ is separable. 
\end{tm}
\begin{proof}
Without loss of generality, we can suppose that the family $\mathcal C$ is a family of disjointed sets that cover $E$ and $\mathcal{C}=\{C_n: n\in \mathbb N\}.$ 
We define $\mathcal F_n=\sigma (C_1,\ldots,C_n).$ For an arbitrary $A \in \mathcal E$ there exists $\{C_{k}: k \in I\}$ such that $C_k \in \mathcal C$ and $I$ is at most countable and $A=\cup_{k \in I}C_k.$
It holds $\mu(A)=\mu\left(\cup_{k \in I}C_k\right)=\sum\limits_{k \in I}\mu(C_k) <\infty.$
If $I$ is finite, we can take $n_{\epsilon}=\max I$ and for $A'=A=\cup_{k \in I}C_k \in \mathcal{F}_{n_{\epsilon}}$ it holds $\mu(A\Delta A')=0<\epsilon.$
If $I$ is countable, we can suppose $I=\{ k_n: n \in \mathbb N.$ Since $\mu(A)=\sum\limits_{n \in \mathbb N}\mu(C_{k_n})<\infty,$ for $\epsilon>0$ there exists $n' \in \mathbb N$ such that $\sum\limits_{n=n'+1}^{\infty}\mu(C_{k_n})<\epsilon.$ If we take $n_{\epsilon}=k_{n'}$ and $A'=\cup_{n=1}^{n'}C_{k_n} \in \mathcal F_{n_{\epsilon}},$ it holds $\mu(A\Delta A')=\sum\limits_{n=n'+1}^{\infty}\mu(C_{k_n})<\epsilon.$ So, $\mu$ is approximable on $\mathcal E$ with respect to finite $(\mathcal F_n)$ and therefore $(\widetilde{\mathcal E},d_{\mu})$ is separable.  
\end{proof}

 \begin{tm}
 \label{tm:decomp}
 Suppose $\mu= \mu_1+ \mu_2$  and suppose that $\mu_2$ is not (uniformly) approximable on $\mathcal E'$ with respect to a  finite filtration, then $\mu$ is not (uniformly) approximable on $\mathcal E'$ with respect to any finite filtration.
 \end{tm} 
 \begin{proof}
 We prove the result for uniformity approximability since the proof in a case of approximability is similar.

 Since $\mu_2$ is not uniformly approximable on $\mathcal E'$ with respect to any  finite filtration, for arbitrary $(F_n)$ and $\epsilon>0$ there exists $B \in \mathcal E'$ such that $\mu_2(A\Delta B)\geq \epsilon$ for all $A \in \cup_{n \in \mathbb N}\mathcal F_n.$ But then
 $$\mu(A\Delta B)=\mu_1(A\Delta B)+\mu_2(A\Delta B)\geq \epsilon,$$
 from which follows that $\mu$ is not uniformly approximable approximable on $\mathcal E'$ with respect to finite filtration.
 
 \end{proof}
   Localizable measures can be decomposed into a non-atomic part and purely atomic part (Theorem \ref{tm: atomic_nonatomic}). Following Theorem \ref{tm:decomp}, measure $\mu$ is (uniformly) approximable if its non-atomic and purely atomic part are (uniformly) approximable. In other words, $(\widetilde{\mathcal E},d_{\mu})$ is separable (compact) if non-atomic and purely atomic part of $\mu$ are (uniformly) approximable.  Therefore, we focus on separability and compactness properties of $(\widetilde{\mathcal E},d_{\mu}),$ first in a case when $\mu$ is non-atomic and then in a case of purely atomic $\mu.$  
\begin{tm}
\label{tm:nonatomic_not_sep}
If $\mu$ is non-atomic measure and $\mu(E)=\infty$, then $(\widetilde{\mathcal E},\mu)$ is not separable.
\end{tm}
\begin{proof}
Suppose that $(\widetilde{\mathcal E},\mu)$ is separable. Then there exists a filtration $(\mathcal F_n),$ $|\mathcal F_n|<\infty$ on $(E,\mathcal E,\mu)$ such that $\mu$ is approximable on $\mathcal E$ with respect to $(\mathcal F_n).$ We can suppose that $\mathcal F_n=\sigma\left(A^{(n)}_1,\ldots,A^{(n)}_{m_n}\right),$ where $A^{(n)}_i \in \mathcal E$ are disjoint  and $\bigcup\limits_{i=1}^{m_n}A^{(n)}_i=E.$ Since $\infty=\mu(E)=\sum\limits_{i=1}^{m_n}\mu(A^{(n)}_i),$ for each $n$ we can find $i_n,$ $1\leq i_n\leq m_n$ such that $\mu(A^{(n)}_{i_n})=\infty$ and since $\mathcal F_n \subseteq \mathcal F_{n+1}$ we can choose $(i_n)_n$ in a way that $A^{(n+1)}_{i_{n+1}}\subseteq A^{(n)}_{i_n},$ $n \in \mathbb N.$ Since $\mu$ is non-atomic, we can construct inductively a sequence of measurable sets $B_n\subset A^{(n)}_{i_n}$ in a following way: $\mu(B_n)=2^{-n}$ and $B_{n+1} \subset A^{(n+1)}_{i_{n+1}}\diagdown B_n.$ Note that $B_n$ are disjointed. Let $B=\cup_{n=1}^{\infty}B_n.$ For each $n \in \mathbb N,$ and an arbitrary $A' \in \mathcal F_n,$ $$\mu(B\Delta A')=\left\{ \begin{array}{ll}
\mu(B)+\mu(A')\geq 1 & A'\cap A^{(n)}_{i_n}=\emptyset,\\
\mu(A^{(n)}_{i_n}\diagdown B)=\infty, & A^{(n)}_{i_n}\subseteq A'.
\end{array}\right.$$
So, $\mu(B\Delta A')\geq 1$ for each $n \in \mathbb N$ which contradicts the assumption of approximability.
\end{proof}
\begin{tm}
\label{tm:atomless_not_compact}
If measure $\mu$ on $(E,\mathcal E)$ is non-atomic  than $(\widetilde{\mathcal E},d_{\mu})$ is not compact or locally compact.
\end{tm}
\begin{proof}
Suppose that $\mu$ is non-atomic, then for each $B \in \mathcal E$ such that $\mu(B)>0,$ there exists $A \in \mathcal E$  such that $A\subseteq B$ and $0<\mu(A)<\mu(B).$ So for each $B \in \mathcal E$ and for each $r \in \mathbb R,$ $0\leq r<\mu(B)$ there exists a measurable set $A\subset B$ such that $\mu(A)=r.$ 

Let  $(\mathcal F_n)_{n \in \mathbb N}$ be an arbitrary filtration on $\mathcal E ,$ such that $\mathcal F_n$ is finite for each $n \in \mathbb N.$   
In this case, we can assume that $\mathcal F_n= \sigma \{ A^{(n)}_i, i \in I_n\}$ where $A^{(n)}_i$ are disjoint, $\cup_{i \in I_n} A^{(n)_i}=\mathbb R^d$ and $I_n$ is a finite set of indices.

Suppose first that $\mu$ is finite. 
Let $0<\epsilon<\frac{\mu(E)}{2}$ and $\alpha=\frac{\epsilon}{\mu(E)}.$ Note that $0<\alpha<\frac{1}{2}.$ 
For each $A^{(n)}_i$ we can find Borel set $C^{(n)}_i$ such that $\mu(C^{(n)}_i)=\alpha\mu(A^{(n)}_i).$ If we define set $C_n=\cup_{i\in \mathbb I_n} C^{(n)}_{i},$  calculation similar to (\ref{eq:Tm_epsilon_diff}) yields $\mu(C_n)=\epsilon$  $\mu(C_n \Delta A)\geq\epsilon$ for each $A \in \mathcal F_n.$  Which shows that $\mu$ is not uniformly approximable on $(F_n)$, so $(\widetilde{\mathcal E},d_{\mu})$ is not compact.

Let $\mu(E)=\infty.$
If $\mu(A_i^{(n)})=\infty$ for all $i \in I_n,$ we can take $C_n=\emptyset.$ Then $\mu(C_n)=0\leq \epsilon$ and $\mu(C_n \Delta A)=\infty\geq \epsilon.$

Otherwise, let $I'_n=\{ i \in I_n: \mu(A_i^{(n)})<\infty).$ 
Let $0<\epsilon<\frac{\mu(\cup_{i \in I'_n}A_i^{(n)})}{2}$ and $\alpha=\frac{\epsilon}{\mu(\cup_{i \in I'_n}A_i^{(n)})}.$
For each $A^{(n)}_i, i \in I'_n$  we take measurable set $C^{(n)}_i$ such that $\mu(C^{(n)}_i)=\alpha\mu(A^{(n)}_i).$ If we define set $C_n=\cup_{i\in \mathbb I'_n} C^{(n)}_{i},$
again we have
$\mu(C_n)=\epsilon$ and $\mu(C_n\Delta A)\geq \epsilon.$

Let $\mathcal E'_{\epsilon}=\{ A \in \mathcal E: \mu(A)\leq \epsilon.\}$
Previous discussion shows that $\mu$ is not uniformly approximable on $\mathcal{E}'_{\epsilon}\subset \mathcal E$ with respect to any filtration containing finite $\sigma$-algebras. We conclude that $(\widetilde{\mathcal E},d_{\mu})$ is not compact, and also each closed ball $\overline{B}([\emptyset],\sqrt{\epsilon})=\{ [B]: B \in \mathcal E'_{\epsilon}\}$ is not totally bounded, but it is also not compact since $\overline{B}([\emptyset],\sqrt{\epsilon})$ is closed. So, $[\emptyset]$ has no compact neighbourhood. We conclude that $(\widetilde{\mathcal E},d_{\mu})$ is not locally compact.   

\end{proof}

From Theorems \ref{tm:nonatomic_not_sep} and \ref{tm:atomless_not_compact} we see that  in case of non-atomic measure $\mu$, $(\widetilde{\mathcal E},d_{\mu})$ is not separable if $\mu$ is an infinite measure and also not compact when $\mu$ is an arbitrary measure.

Further on, let $\mu$ be a purely atomic measure on $(E,\mathcal E).$ Suppose that atoms of the $\mu$  are singletons. If $E$ is second countable Hausdorff and $\mathcal E$ is a Borel $\sigma$-algebra on $E,$ following Lemma \ref{atoms_singl}, every measure $\mu$ on $(E,\mathcal E)$ satisfies the condition.

We define set $E_{fin}=\{ x \in E: 0<\mu(\{x\})<\infty\}$ of all atoms with a finite measure and set $E_{\infty}=\{ x \in E: \mu(\{x\})=\infty\}$ of all atoms with an infinite measure.
To prove that  corresponding $(\widetilde{\mathcal E}_{fin},d_{\mu})$ is separable if and only if $E_{fin}$ is countable we  need the following result.

\begin{lm}
\label{lm:count_sep}
Let $(E,\mathcal E)$ be a measurable space.  Let $(\mathcal F_n)$ be a filtration on $\mathcal E$ such that $\mathcal F_n$ can be represented as $\sigma\left(A^{(n)}_1,\ldots,A^{(n)}_{m_n}\right)$ where $m_n\in \mathbb N$ and $A^{(n)}_1,\ldots,A^{(n)}_{m_n}$ is a finite partition of $E.$ If $E_{fin}$ is a uncountable subset of $E,$ there exists a decreasing sequence $\left(A^{(n)}_{j_n}\right)_n,$ $j_n \in \{ 1,\ldots,m_n\}$ such that $A^{(n)}_{j_n}\cap E_{fin}$ is infinite for each $n \in \mathbb N$ and $\bigcap\limits_{n \in \mathbb N}A^{(n)}_{j_n} \cap E_{fin}\neq \emptyset.$
\end{lm}
\begin{proof}
It holds $E=\bigcap\limits_{n \in \mathbb N}E=\bigcap\limits_{n \in \mathbb N}\bigcup\limits_{j_n=1}^{m_n} A^{(n)}_{j_n}=\bigcup\limits_{(j_n)_n \in \prod_{n=1}^{\infty}\{1,\ldots,m_n\}}\bigcap\limits_{n \in N}A^{(n)}_{j_n}.$ It follows that $E_{fin}=E\cap E_{fin}=\bigcup\limits_{(j_n)_n \in \prod_{n=1}^{\infty}\{1,\ldots,m_n\}}\bigcap\limits_{n \in \mathbb N}A^{(n)}_{j_n}\cap E_{fin}.$
 Note that partition  \break $A^{(n+1)}_1,\ldots,A^{(n+1)}_{m_{n+1}}$ refines partition $A^{(n)}_1,\ldots,A^{(n)}_{m_n},$ so if $\bigcap\limits_{n \in N}A^{(n)}_{j_n}\neq \emptyset$ then $A^{(n+1)}_{j_{n+1}}\subseteq A^{(n)}_{j_n}.$

 Suppose conversely, that for each $\left(A^{(n)}_{j_n}\right)_n,$ $j_n \in \{ 1,\ldots,m_n\}$ such that $A^{(n)}_{j_n}\cap E_{fin}$ is infinite for each $n \in \mathbb N,$  $\bigcap\limits_{n \in \mathbb N}A^{(n)}_{j_n} \cap E_{fin}=\emptyset.$ Then 
 $E_{fin}=\bigcup\limits_{(j_n)_n \in J'}\bigcap\limits_{n \in \mathbb N}A^{(n)}_{j_n}\cap E_{fin},$ where $J'=\{ j=(j_n)_n \in \prod_{n=1}^{\infty}\{1,\ldots,m_n\}: \text{there exists \ } m_{j} \in \mathbb N \text{ \ such that \ }  |A^{(k)}_{j_k}\cap E_{fin}|<\infty, k\geq m_j\}.$ If we denote by $J_{fin}=\{(j_1,\ldots,j_M): j \in J', M\geq m_j \},$ in this case it holds $$E_{fin}\subseteq \bigcup\limits_{ j \in  J'}\bigcap\limits_{n=1}^{m_j}A^{(n)}_{j_n}\cap E_{fin} \subseteq \bigcup\limits_{ (j_1,\ldots,j_M) \in  J_{fin}}\bigcap\limits_{n=1}^{M}A^{(n)}_{j_n}\cap E_{fin}.$$ Since $J_{fin}$ is at most countable (as a subset of all finite sequences $(j_1,\ldots,j_M),$ $M\in \mathbb N,$ $j_n \in 1,\ldots,m_n$) and  $|\bigcap\limits_{n=1}^{M}A^{(n)}_{j_n}\cap E_{fin}|<\infty,$ it follows that $E_{fin}$ is at most countable, which is contradicts the assumption that $E_{fin}$ is uncountable.   
\end{proof}
\begin{tm}
\label{tm:sep_atomic}
Let $\mu$ be a purely atomic measure on $(E, \mathcal E)$ where all the atoms are singletons.
\begin{itemize}
    \item[(a)] $(\widetilde{\mathcal E},d_{\mu})$ is separable  if and only if $\sum\limits_{x \in E_{fin}}\mu(\{x\})<\infty$ and $E_{\infty}$ is finite.
    \item[(b)] $(\widetilde{\mathcal E}_{fin},d_{\mu})$ is separable  if and only if $E_{fin}$ is countable.
\end{itemize}
\end{tm}
\begin{proof}
Suppose  $E_{fin}$ is countable, so it can be written in a form $E_{fin}=\{x_n \in E, \ n \in \mathbb N\}.$ We  first prove that $(\widetilde{\mathcal E}_{fin},d_{\mu})$ is separable. Denote by $\mathcal E_{fin}=\{ B \in \mathcal E: \mu(B)<\infty\}.$ We set $\mathcal{F}_n=\sigma(\{x_1\},\ldots,\{x_n\}),$ $n \in \mathbb N.$ 

For $B \in \mathcal E_{fin}$ it holds $\mu(B)=\mu(B\cap E_{fin})=\sum\limits_{n=1}^{\infty}\mu(B\cap\{x_n\})=\sum_{i \in I_B}\mu(\{x_i\})<\infty,$ where $I_B=\{ n \in \mathbb N:x_n \in B\}.$
For arbitrary $\epsilon>0$ and arbitrary $B \in \mathcal E_{fin}$ there exists finite $I_{\epsilon}\subset I_B $ such that $\sum\limits_{i \in I_B\diagdown I_{\epsilon}}\mu(\{x_i\})=\mu(B \Delta \bigcup\limits_{i \in I_{\epsilon}}\{x_i\} )<\epsilon.$ If we choose $n_{\epsilon}=\max I_{\epsilon},$ $\bigcup\limits_{i \in I_{\epsilon}}\{x_i\} \in \mathcal F_{n_{\epsilon}}.$ We conclude that $\mu$ is approximable on $\mathcal E_{fin}$ with respect to $(\mathcal F_n),$ so $(\widetilde{\mathcal E}_{fin},d_{\mu})$ is separable.

Further on, suppose  $\sum\limits_{x \in E_{fin}}\mu(\{x\})<\infty$ and $E_{\infty}$ is finite, i.e. $E_{\infty}=\{y_1,\ldots,y_m\}$ for fixed $m \in N$ and let \break $\mathcal{F}^{\infty}_n=\sigma(\{y_1\},\ldots,\{y_m\},\{x_1\},\ldots,\{x_n\}),$ $n \in \mathbb N.$ 
For $B \in \mathcal E$ such that $\mu(B)=\infty$ it holds  $B\cap E_{\infty}\neq\emptyset$ and $\mu(B\diagdown E_{\infty})<\infty.$ For $B'=B\diagdown E_{\infty} \in \mathcal E_{fin}$ we have already proven that for $\epsilon>0$ there exists $n$ and $A'\in \mathcal F_n$ such that $\mu(B'\Delta A')<\epsilon.$ If we take $A''=(B\cap E_{\infty})\cup A' \in \mathcal F^{\infty}_n$ it holds $\mu(B\Delta A'') =\mu(B'\Delta A')<\epsilon,$ so $\mu$ is approximable on $\mathcal E$ with respect to $(\mathcal F_n^{\infty})$ and we conclude $(\widetilde{\mathcal E},d_{\mu})$ is separable.

To prove separability of $(\widetilde{\mathcal E}_{fin},d_{\mu})$  implies $E_{fin}$ is countable, we suppose conversely that $(\widetilde{\mathcal E}_{fin},d_{\mu})$  is separable and $E_{fin}$ is uncountable.
Since $(\widetilde{\mathcal E}_{fin},d_{\mu})$ is separable, there exists filtration $(\mathcal F_n)$ such that $\mathcal F_n$ is finite and $\mu$ is approximable on $\mathcal E$ with respect to $(\mathcal F_n).$ Each $\mathcal F_n$ can be represented as $\sigma\left(A^{(n)}_1,\ldots,A^{(n)}_{m_n}\right)$ where $m_n\in \mathbb N$ and $A^{(n)}_1,\ldots,A^{(n)}_{m_n}$ is a finite partition of $E.$

Since $E_{fin}\subseteq E$ is uncountable, following Lemma \ref{lm:count_sep} there exists $(j_n) \in \prod_{n=1}^{\infty}\{1,\ldots,m_n\}$ such that $E_{fin}\cap \left(\bigcap\limits_{n \in N}A^{(n)}_{j_n}\right)$ is uncountable. In other words, there exists a decreasing sequence $(A^{(n)}_{j_n})_n$ such that   $E_{fin}\cap A^{(n)}_{j_n}$ is infinite for each $n \in \mathbb N$ and the intersection $\cap_{n \in \mathbb N}A^{(n)}_{j_n}\cap E_{fin}$ is non-empty. We take $x \in \cap_{n \in \mathbb N}A^{(n)}_{j_n} \cap E_{fin},$ set $B=\{x\}$ and take $\epsilon$ such that $0<\epsilon\leq\mu(\{x\}).$ 
For arbitrary $n\in \mathbb N,$ if $A \in \mathcal F_n$ such that $A^{(n)}_{j_n}\subseteq A$ then $\mu(B\Delta A)\geq \mu(A^{(n)}_{j_n}\diagdown B)\geq\sum\limits_{y \in E_{fin}\cap A^{(n)_{j_n}}}\mu(\{y\})=\infty,$ since the sum is uncountable sum of non-negative numbers. If $A^{(n)}_{j_n}\cap A=\emptyset$ then
$\mu(B\Delta A)\geq \mu(\{x\})=\epsilon.$
So, for each $n \in \mathbb N$ and each $A \in \mathcal F_n$ it holds
\[
\mu(B\Delta A)\geq \epsilon.
\]
This is contradiction to a fact that $\mu$ is  approximable on $\mathcal E_{fin}$ with respect to $(\mathcal F_n)$ and therefore $(\widetilde{\mathcal E}_{fin},d_{\mu})$ is not separable.  

Let us now prove that if  $(\widetilde{\mathcal E},d_{\mu})$ is separable then $E_{fin}$ is countable and $E_{\infty}$ is finite.
We suppose, conversely,  $(\widetilde{\mathcal E},d_{\mu})$ is separable and $E_{fin}$ is uncountable or $E_{\infty}$ is infinite.
If $E_{fin}$ is uncountable then $(\widetilde{\mathcal E}_{fin},d_{\mu})$ is  not separable. Since $(\widetilde{\mathcal E}_{fin},d_{\mu})$ is a subspace of $(\widetilde{\mathcal E},d_{\mu}),$ using Theorem \cite[Theorem VIII, p.~160 ]{Zo} we can conclude $(\widetilde{\mathcal E},d_{\mu})$ is not separable.

If $E_{\infty}$ is infinite, its partitive set is also infinite. So,  for an arbitrary filtration $(\mathcal F_n)$ on $\mathcal E$ with $|\mathcal F_n|<\infty$ we can find $B\subset E_{\infty}$ such that $B \notin \mathcal F_n,$ $n \in \mathbb N.$ It holds $\mu(B\Delta A')=\infty$ for all $A' \in \mathcal{F}_n,$ $n \in \mathbb N.$ So, $\mu$ is not approximable on $\mathcal E$ with respect to any $(\mathcal F_n)$ where $|\mathcal F_n|<\infty$ and therefore $(\widetilde{\mathcal E},d_{\mu})$ is not separable. 
\end{proof}
\begin{tm}
\label{tm:purely_atomic_comp}
Let $\mu$ be purely atomic measure on $(E, \mathcal E)$ where all the atoms are singletons.
\begin{itemize}
    \item[(a)] $(\widetilde{\mathcal E},d_{\mu})$ is compact if and only if $\sum\limits_{x \in E_{fin}}\mu(\{x\})$ is finite and $E_{\infty}$ is finite. 
    \item[(b)] $(\widetilde{\mathcal E}_{fin},d_{\mu})$ is compact if and only if $\sum\limits_{x \in E_{fin}}\mu(\{x\})$ is finite.
\end{itemize}
 
\end{tm}
\begin{proof}

First, let us prove that $\sum\limits_{x \in E_{fin}}\mu(\{x\})<\infty$ implies that $(\widetilde{\mathcal E}_{fin},d_{\mu})$ is compact.
 Since $\sum\limits_{x \in E_{fin}}\mu(\{x\})$ implies that $E_{fin}$ is countable (see e.g. \cite[Theorem 3.12.6., p. 131]{Va}).  We can assume that $E_{fin}=\{ x_n, n\in \mathbb N\}.$ We denote $\mathcal E_{fin}=\{ B \in \mathcal E: \mu(B)<\infty\}.$ It holds
      $$\sum\limits_{n=1}^{\infty}\mu(\{x_n\})<\infty.$$
      So, for every $\epsilon>0$ there exists $n_\epsilon\in \mathbb N$ such that $\sum\limits_{n=n_{\epsilon}}^{\infty}\mu(\{x_n\})<\epsilon.$
      Define $\mathcal F_n=\sigma\left(\{ x_1\},\ldots,\{x_n\}\right).$ Let us show that $\mu$ is uniformly approximable  on $\mathcal E_{fin}$ with respect to $\mathcal{F}_n.$
      If we take arbitrary $B \in \mathcal E_{fin}$ and set $A=B\cap \{x_1,\ldots,x_{n_{\epsilon}}\} \in \mathcal F_{n_{\epsilon}}$ it holds that
      $$\mu(A\Delta B)\leq \sum\limits_{n=n_{\epsilon}}^{\infty}\mu(\{x_n\})<\epsilon,$$
      so following Theorem \ref{tm:comp}, $\widetilde{\mathcal{E}}_{fin}=\{ [B]: B\in \mathcal E_{fin}\}$ is compact.

      Suppose $E_{\infty}=\{y_1,\ldots,y_m\},$ for fixed $m \in \mathbb N.$ To prove $(\widetilde{\mathcal{E}},d_{\mu})$ is compact, we set $\mathcal{F}^{\infty}_n=\sigma(\{y_1\},\ldots,\{y_m\},\{x_1\},\ldots,\{x_n\}),$ $n \in \mathbb N.$
For $B \in \mathcal E$ such that $\mu(B)=\infty$ it holds  $B\cap E_{\infty}\neq\emptyset$ and $\mu(B\diagdown E_{\infty})<\infty.$ We have shown for $\epsilon>0$ there exists $n$ such that for every $B'\in \mathcal E_{fin}$ there exists $A'\in \mathcal F_n$ such that $\mu(B'\Delta A')<\epsilon.$ If we take $A''=(B\cap E_{\infty})\cup A' \in \mathcal F^{\infty}_n$ it holds $\mu(B\Delta A'') =\mu(B'\Delta A')<\epsilon,$ so $\mu$ is  uniformly approximable on $\mathcal E$ with respect to $(\mathcal F_n^{\infty})$ and we conclude $(\widetilde{\mathcal E},d_{\mu})$ is compact.
      
      Let us prove that if $(\widetilde{\mathcal E}_{fin},d_{\mu})$ is compact then $\sum\limits_{x \in E_{fin}}\mu(\{x\})<\infty.$
      Suppose conversely, that $(\widetilde{\mathcal E}_{fin},d_{\mu})$ is compact and $\sum\limits_{x \in E_{fin}}\mu(\{x\})=\infty.$
      If $E_{fin}$ is uncountable, then following Theorem \ref{tm:sep_atomic} $(\widetilde{\mathcal E}_{fin},d_{\mu})$ is not separable, so it cannot be compact.
      Suppose now that $(\widetilde{\mathcal E}_{fin},d_{\mu})$ is compact and $E_{fin}=\{x_n, n\in \mathbb N$ is countable and $\sum\limits_{x\in E_{fin}}\mu(\{x\})=\sum\limits_{n \in \mathbb N}\mu(\{x_n\})=\infty.$   Since $(\widetilde{\mathcal E}_{fin},d_{\mu})$ is compact there exists a filtration $(\mathcal F_n)$ such that $\mathcal F_n$ is finite for each $n \in \mathbb N $ and $\mu$ is uniformly approximable on $\mathcal E_{fin}=\{ B\in \mathcal E: \mu(B)<\infty\}$ with respect to $(\mathcal F_n).$
      For each $\mathcal F_n$ it holds that $\mathcal F_n=\sigma\left(A^{(n)}_1,\ldots,A^{(n)}_{m_n}\right)$ where $m_n\in \mathbb N$ and $A^{(n)}_1,\ldots,A^{(n)}_{m_n}$ is a finite partition of $E.$
      We take an arbitrary $\epsilon>0.$ For each $n$ there exists $A^{(n)}_{j_n}$ such that $A^{(n)}_{j_n}\cap E_{fin}$ is infinite, $A^{(n)}_{j_n}\cap E_{fin}=\{ x'_n:n \in \mathbb N\}$ and $\sum\limits_{n \in \mathbb N}\mu(\{x'_n\})=\infty.$  We an find $m \in \mathbb N$ such that $\sum\limits_{n=1}^m\mu(\{ x'_n\})\geq \epsilon.$ If we take $A=\{x'_1,\ldots,x'_m\} \subseteq A^{(n)}_{j_n}\cap E_{fin},$  then it is easy to see that 
      $\mu(A\Delta A')\geq \epsilon$ for each $A' \in \mathcal F_n$ since if $A'\cap A^{(n)}_{j_n}=\emptyset$ 
      $\mu(A\Delta A')\geq \mu(A)\geq \epsilon$ and if $ A^{(n)}_{j_n}\subseteq A',$ $\mu(A\Delta A')\geq \mu(A^{(n)}_{j_n}\diagdown A)=\sum\limits_{n=m+1}^{\infty}\mu(\{x'_n\})=\infty\geq \epsilon.$
      So, $(\widetilde{\mathcal E}_{fin},d_{\mu})$ cannot be uniformly approximable on $\mathcal E_{fin}$ with respect to $(\mathcal F_n)$ and $(\widetilde{\mathcal E}_{fin},d_{\mu})$ is not compact. 
      
      Let us now prove that if  $(\widetilde{\mathcal E},d_{\mu})$ is compact then $\sum\limits_{x \in E_{fin}}\mu(\{x\})<\infty$ and $E_{\infty}$ is finite.
We suppose, conversely,  $(\widetilde{\mathcal E},d_{\mu})$ is compact and $\sum\limits_{x \in E_{fin}}\mu(\{x\})=\infty$ or $E_{\infty}$ is infinite.
If $\sum\limits_{x \in E_{fin}}\mu(\{x\})=\infty$ then $(\widetilde{\mathcal E}_{fin},d_{\mu})$ is  not compact. Since $(\widetilde{\mathcal E}_{fin},d_{\mu})$ is a closed subspace of $(\widetilde{\mathcal E},d_{\mu}),$ using Theorem \cite[Theorem VIII, p.~160 ]{Zo} we can conclude $(\widetilde{\mathcal E},d_{\mu})$ is not compact.

      If $E_{\infty}$ is infinite, $(\widetilde{\mathcal E},d_{\mu})$ is not separable, so it cannot be compact.

\end{proof}
We conclude the main part of the paper with a few examples. 
\begin{ex}
Let $\mu$ be a counting measure on $(\mathbb N,2^{\mathbb N}),$ i.e. $\mu(A)=|A|$ where $|A|$ stands for the cardinal number of $A$ if $A$ is finite and $\infty$ otherwise. Since $E_{fin}$ is countable,  $(\widetilde{\mathcal E}_{fin},d_{\mu})$ is separable, but it is not compact since $\sum\limits_{x \in E_{fin}}\mu(\{x\})=\infty.$  Also $(\widetilde{\mathcal E},d_{\mu})$ is not  separable.
For arbitrary finite $B$ and arbitrary $0<\epsilon<1,$  $B([B],\epsilon)=\overline{B}([B],\epsilon)=\{[B]\}.$ So, the closed ball with radius less than 1 around each element of $(\widetilde{\mathcal E}_{fin},d_{\mu})$ and $(\widetilde{\mathcal E},d_{\mu})$ is compact since it is finite.  Therefore, they are both locally compact.
\end{ex}
\begin{ex}
Let $\mu$ be a counting measure on $(\mathbb R^d,2^{\mathbb R^d}),$ $\mu(A)=|A|$ where $|A|$ stands for the cardinal number of $A$ if $A$ is finite and $\infty$ otherwise. Since $E_{fin}$ is uncountable, corresponding  $(\widetilde{\mathcal E}_{fin},d_{\mu})$ is not separable and not compact. However, $(\widetilde{\mathcal E},d_{\mu})$ and $(\widetilde{\mathcal E}_{fin},d_{\mu})$ are locally compact.
Similarly to previous example, for arbitrary finite $B$ and arbitrary $0<\epsilon<1,$  $B([B],\epsilon)=\overline{B}([B],\epsilon)=\{[B]\}.$ So, the closed ball with radius less than 1 around each element of $(\widetilde{\mathcal E},d_{\mu})$ is compact since it is finite. 

\end{ex}
\begin{ex}
Let $\mu$ be a counting measure and $\lambda$ a Lebesgue measure on $(\mathbb R^d,\mathfrak{B}(\mathbb R^d)).$ Let $\nu(A)=\mu(A\cap B(0,1))+\lambda(A), A \in \mathfrak{B}(\mathbb R^d),$ where  $B(0,1)$ stands for a unit ball with centre at the origin.  It is easy to see $(\widetilde{\mathcal E}_{fin},d_{\mu})$ and $(\widetilde{\mathcal E},d_{\mu})$ are not separable   not compact (since for $\nu$, $E_{fin}$ is uncountable) and not locally compact. 
\end{ex}
\section{Discussion}
Using newly defined terms of approximability and uniformly approximability, the conditions for separability and compactness of $(\widetilde{\mathcal{E}},d_{\mu})$  and  $(\widetilde{\mathcal{E}}_{fin},d_{\mu})$ can be summarised in a Table \ref{tab:sum}. Table \ref{tab:sum2} provides the topological properties of $(\widetilde{\mathcal{E}},d_{\mu})$  and  $(\widetilde{\mathcal{E}}_{fin},d_{\mu})$ based on finiteness and atomicity properties of the corresponding measure space  $(E,\mathcal E,\mu).$
\begin{table}[H]
    \centering
    \begin{tabular}{|c||c|c|}
    \hline 
        & $(\widetilde{\mathcal{E}},d_{\mu})$  &  $(\widetilde{\mathcal{E}}_{fin},d_{\mu})$ \\ \hline \hline
       Separable & \begin{minipage}{4.5 cm} for $\mathcal E$ countably generated if and only if $\mu$ is finite (Theorem \ref{tm:sep_2} and Theorem \ref{tm:atomless_not_compact}). \end{minipage} & \begin{minipage}{4.5 cm}
       \vspace{0.1 cm}
      
           $\bullet$ for $\widetilde{\mathcal E}_{fin}=\widetilde{\mathfrak B (\mathbb R^d)}_{fin},$ if $\mu$ is outer regular (Theorem \ref{tm:sep_1}),\\
     $\bullet$ for $\mu$ purely atomic measure where all the atoms are singletons, if and only if  the set of atoms with finite measure $E_{fin}$ is countable (Theorem \ref{tm:sep_atomic}).

               \vspace{0.1 cm}\end{minipage}\\ \hline
       Compact  & \begin{minipage}{4.5 cm} if and only if $\mu$ is purely atomic with all atoms singletons and $\sum\limits_{x \in E_{fin}}\mu(\{x\})<\infty$ and and $E_{\infty}$ is finite (Theorem \ref{tm:atomless_not_compact} and Theorem \ref{tm:sep_atomic} (a)).\end{minipage} & \begin{minipage}{4.5 cm}
       \vspace{0.1 cm} for $\mu$ is purely atomic measure where all atoms are singletons if and only if  $\sum\limits_{x \in E_{fin}}\mu(\{x\})<\infty$ (Theorem \ref{tm:purely_atomic_comp}).
       \vspace{0.1 cm}\end{minipage}  \\ \hline

    \end{tabular}
    \caption{Condition on $(E,\mathcal E,\mu)$ for separability and compactness of $(\widetilde{\mathcal{E}},d_{\mu})$ and $(\widetilde{\mathcal{E}}_{fin},d_{\mu})$}
    \label{tab:sum}
\end{table}

\begin{table}[H]

    \centering
    \begin{tabular}{|c||c|c|}
    \hline
    & $\mu$ non-atomic & $\mu$ purely atomic \\ \hline \hline
$\mu$ finite         & \begin{minipage}{4.5 cm} $(\widetilde{\mathcal{E}},\widetilde{\mu})=(\widetilde{\mathcal{E}}_{fin},\widetilde{\mu})$ is:\\
$\bullet$ not compact (Theorem \ref{tm:atomless_not_compact}),\\
$\bullet$ is separable if $\mathcal E$ is countably generated (Theorem \ref{tm:sep_2}). \end{minipage} & \begin{minipage}{4.5 cm} $(\widetilde{\mathcal{E}},\widetilde{\mu})=(\widetilde{\mathcal{E}}_{fin},\widetilde{\mu})$ is:\\
$\bullet$ separable and compact (Theorem \ref{tm:sep_atomic} and Theorem \ref{tm:purely_atomic_comp}).\end{minipage} \\ \hline
    $\mu$ infinite  & \begin{minipage}{4.5 cm} $(\widetilde{\mathcal{E}},\widetilde{\mu})$ is:\\
    $\bullet$ not separable (Theorem \ref{tm:nonatomic_not_sep}),\\
    $\bullet$ not compact (Theorem \ref{tm:atomless_not_compact}).\\
    $(\widetilde{\mathcal{E}}_{fin},\widetilde{\mu})$ is:\\
     $\bullet$ is separable if $\mu$ is outer regular and $(E,\mathcal E)=(\mathbb R^d,\mathfrak{B}(\mathbb R^d).$\\

    \end{minipage}& \begin{minipage}{4.5 cm}
    $(\widetilde{\mathcal{E}},\widetilde{\mu})$ is:\\
    $\bullet$ separable and compact if and only if $\sum\limits_{x \in E_{fin}}\mu(\{x\})$ is finite and $E_{\infty}$ is finite ((Theorem \ref{tm:sep_atomic} (a) and Theorem \ref{tm:purely_atomic_comp} (a)).\\
    $(\widetilde{\mathcal{E}}_{fin},\widetilde{\mu})$ is:\\
    $\bullet$ separable if and only if $E_{fin}$ is countable (Theorem \ref{tm:sep_atomic} (b)),\\
    $\bullet$ compact if and only if $\sum\limits_{x \in E_{fin}}\mu(\{x\})$ is finite (Theorem \ref{tm:purely_atomic_comp} (b)).
    \end{minipage} 
    \\ \hline
    \end{tabular}
    \caption{Separability and compactness of $(\widetilde{\mathcal E},d_{\mu})$ and $(\widetilde{\mathcal E}_{fin},d_{\mu})$ depending whether $\mu$ is finite or infinite, purely atomic or non-atomic. }
    \label{tab:sum2}
\end{table}

\section*{ACKNOWLEDGEMENT}
\small
The author is grateful to  Lev B. Klebanov and Vlasta Matijević for the inspiring discussions and useful remarks and Jakub Stan\v{e}k  for careful reading of the manuscript, helpful comments ad suggestions.\\
Supported in a part by the Czech Science Foundation, project No.~19-04412S, and by the Grant Agency of the Czech Technical University in Prague, project No.~SGS21/056/OHK3/1T/13.



\end{document}